\documentclass[12pt]{amsart}
\usepackage{amsmath}
\usepackage{amssymb}
\usepackage{epsfig}
\usepackage{wasysym}
\usepackage{graphicx}
\usepackage{tikz}
\usepackage{tikz-cd}
\usepackage{bm}
\usepackage{xcolor}
\numberwithin{equation}{section}
\usepackage{extpfeil}

\usepackage[all]{xy}

\input xy
\xyoption{all}

\DeclareFontFamily{U}{mathb}{\hyphenchar\font45}
\DeclareFontShape{U}{mathb}{m}{n}{
      <5> <6> <7> <8> <9> <10> gen * mathb
      <10.95> mathb10 <12> <14.4> <17.28> <20.74> <24.88> mathb12
      }{}
\DeclareSymbolFont{mathb}{U}{mathb}{m}{n}
\DeclareMathSymbol{\righttoleftarrow}{3}{mathb}{"FD}

\calclayout
\allowdisplaybreaks[3]

\theoremstyle{plain}
\newtheorem{prop}{Proposition}[section]

\newtheorem{theo}[prop]{Theorem}
\newtheorem{coro}[prop]{Corollary}

\newtheorem{lemm}[prop]{Lemma}

\theoremstyle{definition}

\newtheorem{prob}[prop]{Problem}

\newtheorem{conj}[prop]{Conjecture}

\def\cB{{\mathcal B}}
\def\cBC{{\mathcal{BC}}}
\def\BC{{\mathcal{BC}}}
\def\SC{{\mathcal{SC}}}

\def\cF{{\mathcal F}}

\def\cH{{\mathcal H}}

\def\cS{{\mathcal S}}

\def\fC{{\mathfrak C}}

\def\fK{{\mathfrak K}}

\def\fS{{\mathfrak S}}

\def\fS{{\mathfrak S}}

\def\bP{{\mathbb P}}
\def\bQ{{\mathbb Q}}

\def\bZ{{\mathbb Z}}
\def\bR{{\mathbb R}}

\def\bN{{\mathbb N}}
\def\bC{{\mathbb C}}

\def\rH{{\mathrm H}}

\def\bF{{\mathbb F}}

\def\res{\mathrm{res}}

\def\BC{\mathcal{BC}}

\def\GL{\mathrm{GL}}

\def\Hom{\mathrm{Hom}}

\def\Burn{\mathrm{Burn}}

\def\lim{\mathrm{lim}}

\def\cd{\mathrm{cd}}

\makeatother
\makeatletter

\begin{document}

\title[Combinatorial Burnside groups]{Combinatorial Burnside groups}

\author{Yuri Tschinkel}
\address{
  Courant Institute,
  251 Mercer Street,
  New York, NY 10012, USA
}

\email{tschinkel@cims.nyu.edu}

\address{Simons Foundation\\
160 Fifth Avenue\\
New York, NY 10010\\
USA}

\author{Kaiqi Yang}
\address{
  Courant Institute,
  251 Mercer Street,
  New York, NY 10012, USA
}
\email{ky994@nyu.edu}

\author{Zhijia Zhang}
\email{zhijia.zhang@cims.nyu.edu}

\date{\today}

\begin{abstract}
We study the structure of combinatorial Burnside groups, which receive equivariant birational invariants of actions of finite groups on algebraic varieties. 
\end{abstract}

\maketitle

\section{Introduction}
\label{sect:intro}

Let $G$ be a finite group, acting regularly on a smooth projective variety over an algebraically closed field $k$, of characteristic zero. The study of such actions, up to $G$-equivariant birationality, is a classical and active area in higher-dimensional algebraic geometry (see, e.g., \cite{serre}, \cite{CS}, \cite{Pro-ICM}). A new type of birational invariants of $G$-actions was introduced in \cite{BnG}. 
These take values in the {\em Burnside group} $$
\Burn_n(G),
$$ 
defined by explicit generators and relations. The invariant is computed on an appropriate birational model $X$ (standard form), where 
\begin{itemize}
\item all stabilizers are abelian,
\item a translate of an irreducible component $Y$ of a locus with nontrivial stabilizer is either equal to $Y$ or is disjoint from it. 
\end{itemize}
The invariant takes into account information about
\begin{itemize}
\item subvarieties $Y\subset X$ with nontrivial (abelian) stabilizers $H$,
\item the induced action of the centralizer $Z_G(H)$ of $H$ on $Y$, and 
\item the representation of $H$ in the normal bundle to $Y$. 
\end{itemize}
A purely combinatorial version of these constructions was introduced in \cite{KT-struct}. It keeps track of the {\em group-theoretic} information extracted as above, while forgetting the {\em field-theoretic} information, i.e., the birational type of the action on irreducible components of loci with nontrivial stabilizers. 

Formally, combinatorial birational invariants of $G$-actions on 
algebraic varieties of dimension $n$ take values in the 
{\em combinatorial Burnside group} 
$$
\cBC_n(G),  
$$
defined in Section~\ref{sect:gen}. 

When $G$ is abelian, there is 
a surjective homomorphism 
$$
\cBC_n(G)\to \cB_n(G),
$$
a group introduced in \cite{KPT} (and in Section~\ref{sect:symb} below), 
which in turn has remarkable arithmetic properties \cite{KPT}, \cite{KT-arith}. For example, 
$$
\cB_n(G)\otimes \bQ = \rH_0(\Gamma(n,G), \cF_n)\otimes \bQ, 
$$
where $\Gamma(n,G)\subset \GL_n(\bZ)$ is 
a certain congruence subgroup  and 
$\cF_n$ is the $\bQ$-vector space generated by characteristic functions of convex rational polyhedral cones in $\bR^n$, modulo
functions of support less than $n$ \cite[Section 9]{KPT}. In particular, the groups $\cB_n(G)$ carry Hecke operators. For $n=2$, there is  a relation between $\cB_2(G)$ and Manin symbols. 

In this note, we investigate arithmetic properties of the {\em a priori} richer groups $\cBC_n(G)$, 
which may be viewed as analogs of Manin symbols for nonabelian groups $G$, and of the {\em finitely generated} ring
$$
\BC_*(G) = \bigoplus_{n\ge 0} \BC_n(G).
$$

Our main result, Theorem~\ref{thm:main}, is the construction of an isomorphism
\begin{equation}
\label{eqn:main}
\BC_n(G) \simeq \bigoplus_{[H,Y]} \cB_n([H,Y]), 
\end{equation}
where the sum is over $G$-conjugacy classes $[H,Y]$ of pairs $(H,Y)$,  with $H\subseteq G$ an abelian subgroup and $H\subseteq Y\subseteq Z_G(H)$, and
$$
\cB_n([H,Y])\simeq \cB_n(H)/(\mathbf{C}_{(H,Y)})
$$
is the quotient by a conjugation relation  
which depends on the representative $(H,Y)$ of the conjugacy class of the pair (see Section~\ref{sect:gen}). For $G$ {\em abelian}, 
we have
$$
\cB_n([H,Y]) = \cB_n(H), \, \text{ and } \,  \BC_n(G) = \bigoplus_{H' \subseteq G} \bigoplus_{H''\subseteq H'} \cB_n(H'');
$$
in particular, the groups $\BC_n(G)$ also carry Hecke operators,  as defined in \cite[Section 6]{KPT} and \cite[Section 3]{KT-arith}.

\

\noindent
{\bf Acknowledgments:} 
We are very grateful to A. Kresch for his interest and comments. 
The first author was partially supported by NSF grant 2000099.

\section{Moebius inversion}
\label{sect:moebius}
Let $G$ be a finite group and $\mathcal H$ the poset of abelian subgroups of $G$ under the inclusion relation. Let $\mathcal S$ be the $\bZ$-module, 
freely generated by $\cH$; we will view $\cH$ as a subset of $\cS$. For $H\in \cH$ we let $(H)$ be its image in $\cS$:
$$
\cS = \bigoplus_{H\in \cH} \bZ (H). 
$$
Let 
$\Psi$ be the $\cS$-valued function on $\cS$, defined on generators by 
$$
\Psi((H))=\sum_{H'\subseteq H} (H'), \quad \forall H\in \cH, 
$$
and extended to all of $\cS$ by $\bZ$-linearity. 
Then there exists a unique $\bZ$-valued function, the {\em Moebius function} 
\begin{equation}
\label{eqn:moe}
\mu=\mu_{\cH} : \cH\times \cH\to \bZ
\end{equation}
such that 
$$
\Phi((H)):=\sum_{H'\subseteq H} \mu(H',H) (H'), \quad \forall H\in \cH, 
$$
is the inverse of $\Psi$, i.e., 
$$
\Psi\circ \Phi = \Phi\circ\Psi =\mathrm{Id}. 
$$

The Moebius function $\mu$ is constructed recursively by rules
\begin{itemize}
    \item $\mu(H,H)=1$, for all $H\in\cH$,
    \item $\mu(H',H)=0$, for all $H',H\in\cH$ with $H'\not\subseteq H$,
    \item 
    $$
    \mu(H'',H)=\displaystyle{-\sum_{H''\subseteq H'\subsetneq H}}\mu(H'', H'),
    $$ 
    for all $H'',H\in\cH$ with $H''\subsetneq H$.
\end{itemize}

When $G$ is {\em abelian} the poset $\cH$ contains all subgroups of $G$ and
is a {\em lattice}, 
with {\em join} and {\em meet} operations defined by 
\begin{align*}
H'\wedge H&:= H'\cap H,\\
\quad H'\vee H&:=\text{subgroup generated by } H' \text{ and } H.
\end{align*}

In Section \ref{sect:struct}, 
we will use the following result concerning the Moebius function on lattices (see, e.g., \cite{Weisner}, or \cite[Sect 5]{Rota}): 

\begin{lemm}
\label{lemm:weisner}
Let $G$ be a finite abelian group and $H'', H'\subseteq G$ subgroups satisfying 
$$
H''\subseteq H'\subsetneq G.
$$
Let $\mu$ be the Moebius function of the subgroup lattice of $G$.
Then 
$$
\sum_{H\subseteq G, \, H\cap H'=H'' }\mu(H,G)=0.
$$ 
\end{lemm}

In Section~\ref{sect:exam}, we will also need a corollary of Lemma~\ref{lemm:weisner}:

\begin{coro}
\label{coro:moebius2}
Let $H,H'$ be subgroups of a fixed finite abelian group, and 
$$
H''\subseteq H\cap H'. 
$$
Then
$$
\sum_{\substack{\widetilde H\subseteq H,  \widetilde H'\subseteq H', \,\widetilde H\cap \widetilde H'= H''}} \mu(\widetilde H,H)\mu(\widetilde H', H')=
\begin{cases}
\mu(H'', H)&\text{if } H=H',\\
0&\text{otherwise}.
\end{cases}
$$
\end{coro}

\begin{proof}
We rearrange the summation
\begin{align*}
\label{eqn:summoebius2}
\sum_{\substack{\widetilde H\subseteq H, \widetilde H'\subseteq H', \, \widetilde H\cap \widetilde H'= H''}} & \mu(\widetilde H,H)\mu(\widetilde H', H')
\\
&=\sum_{\widetilde H\subseteq H}\mu(\widetilde H,H)\sum_{\substack{\widetilde H'\subseteq H',\, \widetilde H'\cap \widetilde H= H''}}\mu(\widetilde H', H'),
\end{align*} 
and apply Lemma~\ref{lemm:weisner}. 
When $H=H'$, the right side 
equals 
$$
\mu(H, H)\mu(H'', H')=\mu(H'', H).
$$
When $H\ne H'$, the right side 
equals 
\begin{align*}
\sum_{H'\subseteq \widetilde H\subseteq H}\mu(\widetilde H, H)\mu(H'', H')=\mu(H'', H')\sum_{H'\subseteq \widetilde H\subseteq H}\mu(\widetilde H, H)=0.
\end{align*}

\end{proof}

\section{Symbols groups}
\label{sect:symb}
Let $G$ be a finite {\em abelian} group, 
$$
G^\vee=\Hom(G,\bC^\times)
$$ 
its character group, and
$$
\cS_n(G) 
$$
the $\bZ$-module generated by $n$-tupels of characters of $G$, 
$$
\beta=(b_1,\ldots, b_n), \quad b_j \in G^\vee, 
\text{ for all } j, 
$$
{\em generating} $G^\vee$, 
modulo the relation

\

\noindent
{\bf (O)} reordering: for all $\beta=(b_1,\ldots, b_n)$ and all $\sigma\in \fS_n$ we have
$$
\beta= \beta^{\sigma}:=(b_{\sigma(1)},\ldots, b_{\sigma(n)}).
$$

\

Consider the quotient 
$$
\cS_n(G) \to \cB_n(G)
$$
by the {\em blowup relation} 

\

\noindent
{\bf (B)}: for all $\beta=(b_1,b_2,\ldots, b_n)$, one has
\begin{equation*}
\label{eqn:rel}
\beta=\beta_1 +\beta_2,
\end{equation*}
where 
\begin{equation}
\label{eqn:beta12}
\beta_1:= (b_1-b_2, b_2, \ldots, b_n), \quad 
\beta_2:= (b_1, b_2-b_1, \ldots, b_n), \quad n\ge 2.
\end{equation}

\

\noindent
For $H\subseteq G$ and $\beta=(b_1,\ldots, b_n)$ we 
put
\begin{equation}
\label{eqn:beta}
\beta|_{H}:=(b_1|_{H}, \ldots, b_n|_{H}).
\end{equation}

The groups $\cB_n(G)$ were introduced in \cite{KPT}; they capture {\em equivariant birational invariants} of $G$-actions on $n$-dimensional algebraic varieties. Combining constructions in \cite{KPT} and \cite{KT-arith}, we know that the groups 
$$
\cB_n(G)_{\bQ}:=\cB_n(G)\otimes \bQ
$$ 
have an interesting internal structure, e.g., they carry:
\begin{itemize}
\item Hecke operators, 
\item multiplication and co-multiplication arising from exact sequences
$$
0\to G'\to G\to 
G''\to 0, 
$$
e.g., multiplication
$$
\nabla : \cB_{n'}(G')_{\bQ}\otimes \cB_{n''}(G'')_{\bQ} \to 
\cB_{n'+n''}(G)_{\bQ}. 
$$
\end{itemize}

\section{Combinatorial Burnside groups}
\label{sect:gen}

\subsection*{Definitions}
Let $G$ be a finite group and $n$ a positive integer. 
The 
{\em combinatorial symbols group} is the  
$\bZ$-module
$$
\mathcal{SC}_n(G), 
$$
generated by triples 
$$
(H, Y, \beta),
$$
where
\begin{itemize}
\item $H\subseteq G$ is an abelian group, 
\item $Y\subseteq G$ is a subgroup satisfying $H\subseteq Y\subseteq Z_G(H)$, and
\item $\beta=(b_1,b_2,\ldots, b_r)$ is a sequence of {\em nontrivial} characters of $H$ of length $r=r(\beta)$, with $1\le r\le n$, {\em generating} $H^\vee$, 
\end{itemize}
subject to relation

\

\noindent
\begin{itemize}
\item[{\bf (O)}] reordering: for all $(H,Y,\beta)$, with 
$\beta=(b_1,\ldots, b_r)$, and $\sigma\in \fS_r$ we have
$$
(H,Y,\beta) = (H,Y,\beta^\sigma), \quad 
\beta^\sigma:= (b_{\sigma(1)},\ldots, b_{\sigma(r)}).
$$
\end{itemize}

\

By convention, 
$$
\SC_0(G):=\bZ.
$$

The {\em combinatorial Burnside group}  is a quotient of
the combinatorial symbols group, 
$$
\SC_n(G)\to \BC_n(G),
$$
obtained by imposing additional relations \cite[Definition 8.1]{KT-struct}:
\begin{itemize}
 \item[($\mathbf{C}$)] conjugation: for all $H$, $Y$, and $\beta$, we have
$$
(H,Y,\beta) = (gHg^{-1}, gYg^{-1}, \beta^g), 
$$
where $\beta^g$ is the image of $\beta$ under the conjugation by $g\in G$, 
\item[($\mathbf{V}$)] vanishing: 
$$
(H, Y, \beta)=0
$$
in each of the following cases: 
\begin{itemize}
\item[$\circ$] $H=1$, 
\item[$\circ$] $b_1+b_2=0$, for some characters $b_1,b_2$ in $\beta$,
\end{itemize}
\item[($\mathbf{B2}$)] blowup relation:


\noindent
for $b_1=b_2$, put:
\begin{equation}
\label{eqn:red}
(H,Y, (b_1,\ldots, b_r)) = (H,Y,(b_2,\ldots, b_r));
\end{equation}
for $b_1\neq b_2$, put:
\begin{align*}
&(H,Y,\beta)=\\
&\begin{cases}
(H, Y, \beta_1) + (H, Y, \beta_2)& \text{ if } 
b_i\in \langle b_1-b_2\rangle, \text{ for some } i,\\  
\underbrace{(H, Y, \beta_1) + (H, Y, \beta_2)}_{\Theta_1} + \underbrace{(\bar{H}, Y, \bar{\beta})}_{\Theta_2}  & \text{ otherwise. } 
\end{cases}
\end{align*}
Here we put
\begin{equation}\label{eqn:Theta2}
\beta_1:=(b_1-b_2, b_2, b_3,\ldots, b_r), \quad \beta_2:=(b_1,b_2-b_1,b_3, \ldots, b_r),
\end{equation}
$$
\bar{H}:=\ker(\langle b_1-b_2\rangle)\subseteq H, \quad \bar\beta:=\beta|_{\bar{H}}.
$$
\end{itemize}
The notation $\Theta_1,\Theta_2$ was used in \cite[Section 4]{BnG} and \cite[Section 2]{KT-struct}.

Relation \eqref{eqn:red} allows to shorten the length of $\beta$ in the presence of repeated characters; we call a symbol {\em reduced} if the characters in $\beta$ are pairwise distinct.


\subsection*{Filtration}
The blowup relation ($\mathbf{B2}$) does not increase $r(\beta)$, the number of characters in $\beta$. This allows to introduce 
$$
\BC_{n,r}(G)\subset \BC_n(G)
$$
as the $\bZ$-submodule generated by reduced symbols where
$\beta$ satisfies 
$r(\beta)\le r$.
We have surjective homomorphisms
$$
\BC_r(G) \to \BC_{n,r}(G), \quad 1\le r\le n, 
$$
which need not be isomorphisms, for $r<n$.

\subsection*{Vanishing}
Relation ($\mathbf{V}$) implies (see \cite[Proposition 4.7]{BnG}) that 
$$
(H, Y,\beta)=0 \in \BC_n(G),
$$
provided there exist  a nonempty $I\subseteq [1,\ldots, r]$
and characters $b_i, i\in I$, such that 
\begin{equation}
\label{eqn:I}
\sum_{i\in I} b_i=0 \in H^\vee.
\end{equation}

\begin{prop}\label{prop:vanish}
For a fixed $G$, we have
$$
\BC_{n}(G)=0, \quad n \gg 0.  
$$
\end{prop}

\begin{proof}
Let $\ell=\ell(G)$ be the maximal order of an element of $G$. 
We have 
$$
0=(H,Y,(\underbrace{b_1,\ldots, b_1}_{\ell {\, {\rm times}}}, b_2, \ldots, b_{n-\ell} )) =
(H,Y, (b_1,b_2, \ldots, b_{n-\ell})) \in \BC_n(G),
$$
for any choices of $b_i$, which implies that  
$$
\BC_{n,r}(G) =0, \quad 1\le r\le n-\ell. 
$$
It suffices to observe that for $n\gg 0$, every 
{\em reduced} symbol has $r(\beta)\le n-\ell$. 
\end{proof}

We define the {\em combinatorial dimension}:
\begin{equation}
\label{eqn:n}
\cd(G):=\min \{ n\in \bN \, |\, \BC_{m}(G) =0 \quad \forall m>n\}.  
\end{equation}
We may also consider versions of this for 
$$
\BC_m(G)\otimes \bQ, \quad \text{  respectively, } \quad \BC_m(G)\otimes \bF_p,
$$
and denote the corresponding smallest $n$ as in \eqref{eqn:n} by
$$
\cd_{\bQ}(G), \quad \text{ respectively, } \quad \cd_p(G).
$$

Computer experiments and Theorem \ref{thm:main} suggest the following:

\begin{conj}
Let $G$ be a finite group, and $H\subseteq G$
a maximal abelian subgroup. Then
$$
\cd(G)\le \log_{2}({|H|}), \quad \text{ and } \quad \cd_{\bQ}(G)\le  \log_{3}({|H|})+1.
$$ 
In particular, for $G=\mathfrak{S}_m$, based on the determination of maximal abelian subgroups of $\fS_n$ in \cite{MaxAbelSn}, we have
$$
\cd(G)\le \frac{m}{3} \log_{2}({3}), \quad \text{ and } \quad 
\cd_{\bQ}(G)\le  \frac{m}{3}+1.
$$ 
\end{conj}

\

\subsection*{Restriction}
Section 7 in \cite{KT-struct} introduced  
the {\em restriction homomorphism};  
in our context, for $G'\subseteq G$, it takes the form: 
$$
\res_{G'}^G: \cBC_n(G)\to \cBC_n(G').
$$
The group $G$ acts by conjugation on the set of generating symbols as in $\mathbf{(C)}$. For any symbol 
$$
\mathfrak s=(H, Y,\beta),
$$
the conjugation action by $G'$ partitions the conjugacy class of $\mathfrak s$ into finitely many orbits. The restriction map is given by 
$$
\mathfrak s\mapsto\sum_{\mathfrak s'}(H'\cap G', Y'\cap G', \beta'\vert_{H'\cap G'}),
$$
where the sum is over orbit representatives $\mathfrak s'=(H', Y', \beta')$. The map respects relations, by construction. It is not
surjective, in general; e.g., 
$$
\cBC_2(\mathfrak{S}_3)=\bZ/2, \quad \cBC_2(\fC_3)=\bZ.
$$

\subsection*{Ring structure}

There is a product map 
$$
 \cBC_{n}(G) \times \cBC_{n'}(G) \rightarrow \cBC_{n+n'}(G),
$$
given as the composition of 
\begin{align*}
    \cBC_{n}(G) \times \cBC_{n'}(G) &\rightarrow \cBC_{n+n'}(G \times G)\\
    (H,Y,\beta) \times (H',Y',\beta') &\mapsto (H\times H',Y\times Y',\beta \cup \beta')
\end{align*}
with restriction to the diagonal. 

We obtain a finitely-generated graded ring 
$$
\cBC_*(G):=\oplus_{n \geq 0}\,\, \BC_n(G),
$$
with $\cBC_0(G)=\bZ$,
subject to various functoriality properties.

\section{Structure theory}
\label{sect:struct}

In this section, we establish an isomorphism of $\BC_n(G)$ with a simpler quotient of the combinatorial symbols group 
$$
\SC_n(G)\to \BC_n'(G),
$$
defined via relations $(\mathbf{C}), (\mathbf{V})$, together with the following modification of the blowup relation:

\

\noindent
\begin{itemize}
\item[$(\mathbf{B2'})$] 
for $b_1=b_2$, put:
\begin{equation}
\label{eqn:redprime}
(H,Y, (b_1,b_2,\ldots, b_r)) = (H,Y,(b_2,\ldots, b_r));
\end{equation}
for $b_1\neq b_2$, put:
$$
(H,Y,\beta)=(H, Y, \beta_1) + (H, Y, \beta_2),
$$
where $\beta_1,\beta_2$ are as in \eqref{eqn:Theta2}.
\end{itemize}

\

For clarity, we will write 
$$
(H,Y,\beta)',
$$
when we view the corresponding symbol as an element in $\BC_n'(G)$.

The relations respect the $G$-conjugacy class $[H,Y]$ of the pair $(H,Y)$; so that 
\begin{equation}
\label{eqn:deco}
\BC_n'(G) =\bigoplus_{[H,Y]}\,\,\, \cB_n([H,Y]), 
\end{equation}
where 
$$
\cB_n([H,Y]):=\bigoplus_{H',Y',\beta'} \,\,\bZ (H',Y',\beta')  /  (\mathbf{C}), (\mathbf{V}), (\mathbf{B2'}), \quad (H',Y') \in [H,Y].
$$

Consider the following conjugation relation on $\cS_n(H)$: 

\

\noindent
$(\mathbf{C}_{(H,Y)})$: for all $\beta \in \cS_n(H)$ and $g\in N_G(H)\cap N_G(Y)$ we have
$$
\beta=\beta^g. 
$$

\

\begin{lemm}
We have have an isomorphism of abelian groups
\begin{align}\label{eqn:normalizeriso}
\cB_n([H,Y])\simeq \cB_n(H) / (\mathbf{C}_{(H,Y)}).
\end{align}
\end{lemm}

\begin{proof}
Fix a representative $(H,Y)$ of the conjugacy class and
consider 
$$
(H',Y', \beta'), \quad (H',Y')\in [H,Y],\quad 
\beta'=(b_1',\ldots,b_r').
$$
Let $g\in G$ be such that
$$
H= gH'g^{-1}\quad\text{and}\quad
    Y= gY'g^{-1},
$$
and put
$$
(b_i')^{g}:= \text{image of } b_i' \text{ under the conjugation by } g, \quad i=1,\ldots,r.
$$
Consider the homomorphism given on symbols in $\cB_n([H,Y])$ by 
\begin{align}\label{eqn:normalizermap}
(H',Y',\beta')\mapsto ((b_1')^{g},\ldots,(b_r')^{g}, \underbrace{0, \ldots, 0}_{n-r})\in \cS_n(H).
\end{align}
This is independent of the choice of $g$: given 
$
g, g'\in G,
$ 
such that  
$$
H'= g^{-1}Hg=g'^{-1}Hg'\quad\text{and}\quad
    Y'= g^{-1}Yg=g'^{-1}Yg',
$$
we have
$$
g'g^{-1}\in N_G(H)\cap N_G(Y).
$$
Therefore, by definition of $(\mathbf{C}_{(H,Y)})$, we have
$$
 ((b_1')^{g},\ldots,(b_r')^{g}, \underbrace{0, \ldots, 0}_{n-r})=((b_1')^{g'},\ldots,(b_r')^{g'}, \underbrace{0, \ldots, 0}_{n-r}),
$$
since 
$$
((b_i')^{g})^{(g'g^{-1})}=(b_i')^{g'},\quad i=1,\ldots, r.
$$
The mapping \eqref{eqn:normalizermap} respects $\mathbf{(C)}$, by construction. Indeed, for any $g\in G$, the symbols 
$$
(H', Y', \beta')\quad\text{and}\quad (gH'g^{-1}, gY'g^{-1},\beta'^g)
$$
will be mapped to the same element in $\cB_n(H)/(\mathbf{C}_{(H,Y)})$
since 
$$
H=g'H'g'^{-1}\quad\text{implies}\quad H=g'g^{-1}\cdot(gH'g^{-1})\cdot gg'^{-1}.
$$
To see its compatibility with $\mathbf{(V)}$ and $\mathbf{(B2')},$ it suffices to observe that the conjugation action is linear, i.e.,
$$
(b_1+b_2)^g=b_1^g+b_2^g,\quad\text{ for all } b_1,b_2\in H^{\vee}, \,g\in G,
$$
and we have the following identities in $\cB_n(H)$, by definition:
\begin{itemize}
    \item $(b_1, b_1, b_2,\ldots)=(0,b_1,b_2,\ldots)$,
    \item $(b_1,b_2,\ldots)=(b_1-b_2, b_2,\ldots)+(b_1,b_2-b_1,\ldots)$,
    \item $(b_1, -b_1,\ldots)=0.$
\end{itemize}
On the other hand, by conjugation relations, the map defined by
$$
(b_1,\ldots,b_n) \mapsto (H,Y,\beta),
$$
where $\beta$ is obtained by removing all 0's in the sequence of $b_i$, is the inverse of the map  \eqref{eqn:normalizermap}. It is clearly compatible with $\mathbf{(B)}$ and $\mathbf{(C}_{(H,Y)})$. Therefore 
\eqref{eqn:normalizermap} induces the desired isomorphism \eqref{eqn:normalizeriso}.
\end{proof}

\

We will now construct an isomorphism of $\bZ$-modules
$$
\BC_n(G)\simeq \BC_n'(G). 
$$
The decomposition \eqref{eqn:deco} above allows us to efficiently compute $\BC_n(G)$, 
and to import further structures into $\BC_n(G)$.

We start by defining a poset relation on the set of symbols: 
$$
\mathfrak s' := (H',Y', \beta') \le (H,Y, \beta)=:\mathfrak s 
$$
if and only if
\begin{itemize}
\item $Y=Y'$,
\item $H'\subseteq H$, and 
\item $\beta'=\beta|_{H'}$.
\end{itemize}
We observe that the {\em intervals} in this poset relation are isomorphic, as posets, to intervals in the poset $\mathcal H$ of abelian subgroups of $G$. Locally, these intervals are isomorphic to intervals in posets of subgroups of finite abelian groups; the corresponding Moebius function is 
the one in \eqref{eqn:moe}.

Consider the following homomorphisms of $\bZ$-modules
$$
\Psi, \Phi: \SC_n(G) \to \SC_n(G)
$$
defined on symbols by
$$
\Psi: (H,Y,\beta)\mapsto \sum_{H'\subseteq H} (H',Y,\beta')',
$$
respectively, 
$$
\Phi: (H,Y,\beta)'\mapsto \sum_{H'\subseteq H} \mu(H',H) (H',Y,\beta'),
$$
where 
$$
\beta'=\beta|_{H'},
$$
%
and extended by linearity. By convention, if $\beta'$ contains a zero, the symbol is 
considered to be zero. 

These are isomorphisms (see Section~\ref{sect:moebius}), we have
$$
\Psi\circ \Phi = \Phi\circ \Psi =\mathrm{Id}.
$$

\begin{theo}
\label{thm:main}
For all $n\ge 1$ and all $G$, the homomorphism $\Psi$ descends to the respective quotients of the combinatorial symbols group, yielding a commutative diagram of abelian groups

\

\centerline{
\xymatrix{
\SC_n(G) \ar[r]^{\Psi} 
\ar[d]_{(\mathbf{C}),(\mathbf{V}), (\mathbf{B2})} & \SC_n(G)\ar[d]^{(\mathbf{C}),(\mathbf{V}), (\mathbf{B2'})} \\
\BC_n(G)\ar[r]^{\Psi} & \BC_n'(G),
}
}

\

\noindent
with an isomorphism on the bottom row, whose  inverse is given by $\Phi$. 
\end{theo}

\begin{proof}
It is clear that both $\Psi$ and $\Phi$ respect relations $\mathbf{(C)}$ and $\mathbf{(V)}$. 
It remains to show their compatibility with $(\mathbf{B2})$, respectively, $\mathbf{(B2')}$. 

First, we show $\Psi$ is compatible with $\mathbf{(B2')}$, i.e. for any symbol
$$
\mathfrak s:=(H, Y, \beta),\quad \beta:=(b_1,\ldots,b_r),
$$
we have 
$$
\Psi(\mathfrak s)\stackrel{?}{=}\Psi((H, Y, \beta_1))+\Psi((H, Y, \beta_2))+\Psi(\Theta_2(\mathfrak s)) \in \BC_n'(G),
$$
with $\beta_1$, $\beta_2$ and $\Theta_2$ defined in \eqref{eqn:Theta2}. Assume $b_1\ne b_2$ and put 
$$
\bar H=\ker(b_1-b_2),
$$ 
then
$$
\Psi((H, Y, \beta_i))=\sum_{\substack{H'\subseteq H, H'\not\subseteq \bar H}}(H', Y,\beta_i|_{H'})',\quad i=1,2,
$$
since when $H'\subseteq \bar H$, the restriction of $\beta_1$ and $\beta_2$ to $H'$ will have nontrivial space of invariants 
(i.e., a zero in the sequence of characters). 

On the other hand, by definition, we have
\begin{align}\label{eqn:splitsum}
    \Psi(\mathfrak s)&=\sum_{\substack{H'\subseteq H, H'\not\subseteq \bar H}}(H', Y,\beta|_{H'})'+\sum_{\substack{H'\subseteq H\cap \bar H}}(H', Y,\beta|_{H'})'.
\end{align}
Observe that
$$
b_1\vert_{H'}=b_2\vert_{H'}\Leftrightarrow H'\subseteq \bar H.
$$ 
Applying $\mathbf{(B2')}$ to the right side of \eqref{eqn:splitsum} yields
\begin{align*}
    \Psi(\mathfrak s)=&\sum_{\substack{H'\subseteq H, H'\not\subseteq \bar H}}(H', Y,\beta_1|_{H'})'+\sum_{\substack{H'\subseteq H, H'\not\subseteq \bar H}}(H', Y,\beta_2|_{H'})'\\
    &+\sum_{\substack{H'\subseteq \bar H}}(H', Y, (b_2|_{H'}, b_3|_{H'},\ldots, b_r'|_{H'}))'\\
    =&\Psi((H, Y, \beta_1))+\Psi((H, Y, \beta_2))+\Psi(\Theta_2(\mathfrak s)).\\
\end{align*}


We now show that $\Phi$ respects $\mathbf{(B2')}$. By definition, we have
\begin{align}\label{eq:mbsum}
    \Phi((H, Y, \beta)')=\sum_{H'\subseteq H}\mu(H', H)(H', Y, \beta\vert_{H'}) \in \BC_n(G).
\end{align}
Consider all $\Theta_2$ terms in $\mathbf{(B2)}$ arising from symbols in the sum on the right side of \eqref{eq:mbsum}:
 \begin{align*}
    &\sum_{H'\subseteq H}\mu(H',H)\Theta_2(H', Y, \beta\vert_{H'})\\
    =&\sum_{H''\subseteq H}\left(\sum_{\bar H\cap H'=H''}\mu(H',H)\right)\cdot(H'',Y, \bar\beta\vert_{H''}).
\end{align*}  
It suffices to observe that $\mu(H', H)$ equals the corresponding value of the Moebius function of the subgroup 
{\em lattice} of the abelian group $H$. Therefore, the compatibility of $\Phi$ with $\mathbf{(B2)}$ 
reduces to Lemma~\ref{lemm:weisner}. 
\end{proof}

\section{Examples and applications}
\label{sect:exam}

\subsection{Abelian groups}
Classification of abelian subgroups of the plane Cremona group, i.e., 
of actions of abelian groups on rational surfaces, is well-understood (see \cite{blancthesis}, and the references therein). Much less is known in higher dimensions. First applications of the Burnside group formalism to the classification of such actions, in particular to actions of cyclic groups on cubic fourfolds, can be found in \cite{HKTsmall}.

When $G$ is abelian, Theorem~\ref{thm:main}, combined with decomposition \eqref{eqn:deco}, shows that
\begin{equation}
\label{eqn:abel}
\BC_n(G) = \bigoplus_{H' \subseteq G} \bigoplus_{H''\subseteq H'} \cB_n(H'').
\end{equation}
For elementary abelian $p$-groups  $G\simeq \bF_p^r$
$$
\cB_n(G)=0, \quad n<r, 
$$
since a sequence of characters of length $<r$ cannot generate the character group.
By \eqref{eqn:abel}, the computation of $\cBC_n(G)$ reduces to 
the computation of $\cB_n(H'')$, where $H''=\bF_p^m$ with $m \leq n$.
The number of $H'' \subseteq G$ such that $H''\cong \bF_p^m$ is
$\# \mathrm{Gr}(m,r)(\bF_p).$
Results in \cite[Section 5]{KPT}, especially Theorem 14, yield finer structural information about $\cB_n(H'')\otimes \bQ$.

\begin{prob}
Determine the ring structure of $\BC_*(G)$, where $G=\bF_p^r$.
\end{prob}

The isomorphisms $\Phi$ and $\Psi$ induce a ring structure on
$$
\BC'_*(G):=\bigoplus_{n\geq0}\BC'_n(G),
$$
with the product map
defined on symbols by
\begin{align}\label{eqn:BC'prod}
(H, Y, \beta)'\widetilde{\times}(H', Y', \beta')'\mapsto \Psi(\Phi((H, Y, \beta)'){\times}\Phi((H', Y', \beta')')).
\end{align}
By construction, $\Psi$ and $\Phi$ are ring isomorphisms 
$$
\cBC_*(G)\simeq\cBC_*'(G).
$$

\begin{prop}
When $G$ is abelian, the product \eqref{eqn:BC'prod} takes the form:
$$
(H, Y, \beta)'\widetilde{\times}(H', Y', \beta')'\longmapsto
\begin{cases}
0& \text{if } H\ne H',\\
(H, Y\cap Y', \beta\cup\beta')'&\text{otherwise}.
\end{cases}
$$
\end{prop}

\begin{proof}
When $G$ is abelian, the conjugacy relation plays no role. For any generating symbols
$$
(H, Y, \beta)',\quad (H', Y', \beta')'\in\BC'_*(G)
$$
we have, by definition, 
\begin{align*}
    & \Phi((H, Y, \beta)')\times\Phi((H', Y', \beta')')\\
    =&\left(\sum_{\widetilde H\subseteq H}\mu(\widetilde H, H)(\widetilde H,Y,\beta|_{\widetilde H})\right){\times}\left(\sum_{\widetilde H'\subseteq H'}\mu(\widetilde H', H')(\widetilde H',Y',\beta'|_{\widetilde H'})\right)\\
    =&\sum_{\substack{\widetilde H\subseteq H,\widetilde H'\subseteq H'}}\mu(\widetilde H,H)\mu(\widetilde H', H')\cdot(\widetilde H\cap \widetilde H', Y\cap Y',(\beta\cup\beta')|_{\widetilde H\cap \widetilde H'})\\
    =&\sum_{H''\subseteq H\cap H'}\left(\sum_{\substack{\widetilde H\subseteq H, \widetilde H'\subseteq H'\\\widetilde H\cap \widetilde H'=H''}}\mu(\widetilde H,H)\mu(\widetilde H', H')\right)\cdot(H'', Y\cap Y',(\beta\cup\beta')|_{H''})\\[0.2cm]
    =&\begin{cases}\displaystyle{\sum_{\widetilde H\subseteq H}\mu(\widetilde H, H)(\widetilde H, Y\cap Y',(\beta\cup\beta')|_{\widetilde H})}&\text{if }  H=H',\\
    0&\text{otherwise,}
    \end{cases}
\end{align*}
where the last equality follows from Corollary~\ref{coro:moebius2}. Applying $\Psi$ to the equality above completes the proof.
\end{proof}

\subsection{Central extensions of abelian groups}
According to \cite{BT}, over $k=\bar{\bF}_p$, the quotient spaces $V/G$
are {\em universal} for {\em unramified cohomology}: given a variety $X/k$ and an unramified class $\alpha\in \rH^i_{nr}(k(X))$ (Galois cohomology with torsion coeffients, coprime to $p$), there exists a birational map $X\to V/G$, where $V$ is a faithful representation of a central extension $G$ of an abelian group, such that $\alpha$ is induced from $V/G$. There is a general algorithm to compute the class, in $\Burn_n(G)$, of $G$-actions on $n$-dimensional linear representations $V$ of $G$, based on De Concini--Procesi models of subspace arrangements \cite{KT-vector}. This motivates the study of $\BC_*(G)$ for groups of such type.

As a first example, let $G=\mathfrak{D}_p$ be the dihedral group of order $2p$, with $p\geq 5$ is a prime. Computer experiments suggest that
$$
\cBC_2(G)=\cB_2([\fC_p, \fC_p])= \mathbb{Z}^{\frac{(p-5)(p-7)}{24}}\times(\mathbb{Z}/2)^{\frac{p-3}{2}}\times\mathbb{Z}/\tfrac{p^2-1}{12}.
$$
The conjugation action on $\beta=(b_1,b_2)$ for symbols in $\cB_2([\fC_p, \fC_p])$ 
is equivalent to
$$
(\fC_p, \fC_p, (b_1,b_2))=(\fC_p, \fC_p, (-b_1, -b_2)).
$$
This leads to a variant of the group 
$\cB_2^-(\fC_p)$ introduced in \cite{KPT}. In fact,
$$
\cB_2^-(\fC_p)\otimes\bQ\simeq\cB_2([\fC_p, \fC_p])\otimes\bQ,
$$
since according to \cite[Proposition 3.2]{HKTsmall}, we have
$$
(\fC_p, \fC_p, (a,b))+(\fC_p, \fC_p, ( -a, b))=0\in\cB_2([\fC_p, \fC_p])\otimes\bQ.
$$
The rank of the torsion-free part of $\cB_2([\fC_p, \fC_p])$ is thus related to the modular curve $X_1(p)$ (see \cite[Section 11]{KPT}).

We may also consider central extensions
\begin{align*}
    0 \rightarrow \mathbb{Z}/p \rightarrow G \rightarrow (\mathbb{Z}/p)^2 \rightarrow 0
\end{align*}
with $Z(G)\cong \mathbb{Z}/p$, and $p$ a prime.
For example, we have 
\begin{itemize}
    \item $p=2$, $G=\mathfrak{D}_4$, $\cBC_2(G)=(\mathbb{Z}/2)^3$.
    \item $p=3$, $G=\mathfrak{He}_3$, $\cBC_2(G)=\mathbb{Z}^{26}$, $\cBC_3(G)=\mathbb{Z}^4$.
    \item $p=5$, $G=\mathfrak{He}_5$, $\cBC_2(G)=\mathbb{Z}^{124}$, $\cBC_3(G)=(\mathbb{Z}/2)^{36}\times \mathbb{Z}^{36}$.
\end{itemize}
According to the structure of the Heisenberg group $\mathfrak{He}_p$, for odd primes $p$, we have 
\begin{align*}
    \cBC_n(\mathfrak{He}_p)=\cB_n([\mathbb{Z}/p,\bZ/p])^{3p+5}\oplus 
    \cB_n([(\mathbb{Z}/p)^2,(\mathbb{Z}/p)^2])^{p+1}.
\end{align*}

\subsection{Symmetric groups}
We compute the combinatorial Burnside groups for small symmetric groups 
$G=\fS_n$:

{\tiny
$$
\begin{tabular}{c|l|l}
$n$  & $\cBC_2(G)$& $\cBC_3(G)$  \\ 
\hline 
$3$ & $\bZ/2$ & 0 \\
\hline
$4$ & $(\bZ/2)^3$ & 0 \\
\hline
$5$ & $(\bZ/2)^6 \times \bZ/4$ & 0 \\
\hline
$6$ & $(\bZ/2)^{31}\times(\bZ/4)^3\times \bZ/8$ & $(\bZ/2)^5\times \bZ/4$ \\
\hline
$7$ & $(\bZ/2)^{57}\times (\bZ/4)^{12}\times (\bZ/8)^2 \times \bZ/3$ & $(\bZ/2)^{16} \times \bZ/4$ \\
\hline
$8$ & $(\bZ/2)^{290}\times(\bZ/4)^{30} \times(\bZ/8)^6\times\bZ/{16}\times (\bZ/{3})^2\times \bZ$ & $(\bZ/2)^{122}\times(\bZ/4)^4 \times \bZ/8 \times \bZ$ \\
\end{tabular}
$$
}

For example, for $G=\mathfrak{S}_4$, the only conjugacy classes $[H,Y]$ that contribute to $\BC_2(G)$ are (the conjugacy classes of) the pairs: 
\begin{enumerate}
    \item $(\fC_3, \fC_3)$, with $\fC_3=\langle (2,4,3) \rangle$,
    \item $(\fK_4,\fK_4)$, with $\fK_4=\langle (3,4),(1,2)(3,4)\rangle$,
    \item $(\fC_4,\fC_4)$, with $\mathfrak{C}_4=\langle (1,4,2,3)\rangle$.
\end{enumerate}
We have 
$$
\cB_2([H,Y]) = \bZ/2
$$ 
for the corresponding summands  of $\BC_2'(G)$.

\subsection{Nonabelian subgroups of the plane Cremona group}
\label{sect:further}

Here, we compute $\BC_*(G)$ for groups admitting {\em primitive} actions on $\bP^2$, namely: 
$$
\mathfrak A_5, \mathsf{ASL}_2(\bF_3), \mathsf{PSL}_2(\bF_7), \mathfrak A_6.
$$

\

\begin{itemize}
    \item $G=\mathfrak{A}_5=\langle (1,2,3)(3,4,5)\rangle \subset \mathfrak{S}_5:$ Nontrivial terms arise from 
    \begin{itemize}
        \item
        $(\fC_3,\fC_3)$, with $\mathfrak{C}_3=\langle (1,2,5) \rangle$,
        \item
        $(\fC_5,\fC_5)$, with $\mathfrak{C}_5=\langle(1,4,5,3,2)\rangle$, 
    \end{itemize}
which contribute
    \begin{itemize}
        \item $\cB_2([(\mathfrak{C}_3,\fC_3)])=\mathbb{Z}/2$,
        \item $\cB_2([(\mathfrak{C}_5,\fC_5)])=(\mathbb{Z}/2)^2$.
    \end{itemize} 
    We have
$$
\cBC_2(G)=(\mathbb{Z}/2)^3, \quad \text{ and } \quad \cBC_n(G)=0, \quad n\ge 3.
$$
\item $G=\mathfrak{C}_3^2:\mathsf{SL}_2(\mathbb{F}_3)=\mathsf{ASL}(2,3)\subset 
\fS_9$, generated by
    \begin{align*}
        \langle (2,5,8)(3,9,6),(2,4,3,7)(5,6,9,8),
        (1,2,3)(4,5,6)(7,8,9)\rangle.
    \end{align*}
We have
$$
\cBC_2(G)=(\mathbb{Z}/2)^7\times \mathbb{Z}^{13}, \quad 
\cBC_3(G)=\mathbb{Z}/2 \times \mathbb{Z}, \quad
\cBC_n(G)=0, n\ge 4.
$$
    \item $G=\mathsf{PSL}(2,7)= \langle (3,6,7)(4,5,8),(1,8,2)(4,5,6)\rangle \subset \mathfrak{S}_8$: 
Nontrivial terms arise from
    \begin{itemize}
        \item 
        $(\mathfrak{C}_3,\fC_3)$, with $\mathfrak{C}_3=\langle(2,6,5)(3,7,4)\rangle$, 
        \item
        $(\mathfrak{C}_7,\fC_7)$, with $\mathfrak{C}_7=\langle (1,2,5,3,6,7,4)\rangle$,
        \item 
        $(\mathfrak{C}_4,\fC_4)$, with $\mathfrak{C}_4=\langle (1,3,4,8)(2,7,6,5)\rangle$,
    \end{itemize}
which contribute 
    \begin{itemize}
        \item $\cB_2([(\mathfrak{C}_3,\fC_3)])=\mathbb{Z}/2$,
        \item $\cB_2([(\mathfrak{C}_7,\fC_7)])=\mathbb{Z}/2 \times \mathbb{Z}$,
        \item $\cB_2([(\mathfrak{C}_4,\fC_4)])=\mathbb{Z}/2$.
    \end{itemize}
We have
$$
\cBC_2(G)=(\mathbb{Z}/2)^3\times \mathbb{Z}, \quad 
\cBC_3(G)=\mathbb{Z}/2, \quad \cBC_n(G)=0, n\ge 4. 
$$
    \item $G=\mathfrak{A}_6 = \langle (1,2)(3,4,5,6),(1,2,3)\rangle$:
We have    
$$    
\cBC_2(G)=(\mathbb{Z}/2)^7\times \mathbb{Z}/4\times \mathbb{Z}, \quad
\cBC_3(G)=\mathbb{Z}/2 \times \mathbb{Z}, \quad \cBC_n(G)=0, n\ge 4.
$$
\end{itemize}

\subsection{A geometric application}
Consider 
$$
G=\mathfrak{C}_2\times \mathfrak{S}_3=\mathfrak{D}_6 = 
\langle (1,2,3,4,5,6),(1,6)(2,5)(3,4)\rangle \subset \mathfrak{S}_6.
$$
    It is known that the linear action of $G$ on $\bP^2$ and the toric action of $G$ on the del Pezzo surface $X$ of degree 6 are not equivariantly birational \cite{isk-s3}. The proof in \cite{isk-s3} relies on tools of the equivariant Minimal Model Program for surfaces, in particular, on the classification of Sarkisov links. 
    
    In \cite[Section 7.6]{HKTsmall}, we used the Burnside group $\Burn_2(G)$ to distinguish these actions. Here, we rework this example in the framework of combinatorial Burnside groups (see also \cite[Section 6]{Burntoric}).

    We have
    $$
    \cBC_2(G)=(\mathbb{Z}/2)^5\times \mathbb{Z}/4,
    $$
    with decomposition
    \begin{itemize}
        \item 
        $H_1=\mathfrak{C}_3=\langle (1,3,5)(2,4,6)\rangle$,
        \item 
        $H_2=\mathfrak{C}_2^2=\langle (2,6)(3,5),(1,4)(2,5)(3,6)\rangle$,
        \item 
        $H_3=\mathfrak{C}_6=\langle (1,2,3,4,5,6)\rangle$.
    \end{itemize}
    Nontrivial contributions to $\BC_2'(G)$ arise from
    \begin{itemize}
        \item $\cB_2([(H_1,H_1)])=\mathbb{Z}/2$,
        \item $\cB_2([(H_2,H_2)])=(\mathbb{Z}/2)^2$,
        \item $\cB_2([(H_1,H_3)])=\mathbb{Z}/2$,
        \item $\cB_2([(H_3,H_3)])=\mathbb{Z}/2\times \mathbb{Z}/4$.
    \end{itemize}

    By \cite[Proposition 6.1]{Burntoric}, 
    we have a formula for the difference 
    $$
    [X\mathrel{\righttoleftarrow}G]-[\mathbb{P}^2\mathrel{\righttoleftarrow}G] \in \rm{Burn}_2(G), 
    $$
    where $\mathbb{P}^2=\bP(1\oplus V_{\chi})$, and $V_{\chi}$ is the standard 
    2-dimensional representation of $\fS_3$, twisted by the character of $\mathfrak C_2$.
    Applying the homomorphism 
    \begin{align*}
    \rm{Burn}_2(G) &\to \cBC_2(G), 
    \end{align*}
 defined in \cite[Proposition 8]{KT-struct}, we obtain the class
    \begin{align*}
    & (diagonal\ in\ \fC_2 \times \mathfrak{S}_2,\fC_2 \times \mathfrak{S}_2,(1)) \\
    +& (\fC_2,\fC_2\times\mathfrak{S}_2,(1))+(\fC_3,\fC_3,(1,1))\\
    -&(\fC_2,\fC_2\times \mathfrak{S}_3,(1))-(\fC_2\times \fC_3,\fC_2\times \fC_3,((0,1),(1,2)))\in \BC_2(G).
    \end{align*} 
    Its image under the map $\Psi$ equals 
 \begin{align*}
   (\fC_3,\fC_3,(1,1))
    -(\fC_3,\fC_2\times \fC_3,(1,2))\in \BC'_2(G), 
    \end{align*}
a nontrivial $2$-torsion class.     
On the other hand, $\cBC_3(G)=0$; in particular, we cannot distinguish the classes of $X\times \bP^1$ and $\bP^2\times \bP^1$, with trivial action on the $\bP^1$-factor. This problem was raised in \cite[Remark 9.13]{lemire}.

\bibliographystyle{plain}
\bibliography{bcn}
\end{document}